\documentclass[11pt]{amsart}

\usepackage{amssymb,amsthm,amsmath}
\usepackage{mathrsfs}
\usepackage{mathtools}
\usepackage[bookmarks]{hyperref}
\usepackage{latexsym}
\usepackage{enumerate}

\usepackage[all]{xy}

\usepackage[T1]{fontenc}

\usepackage{indentfirst}
\usepackage{color}
\usepackage{graphicx}
\usepackage{verbatim}
\usepackage{comment}
\usepackage{bm}
\usepackage{tikz}
\usepackage{tikz-cd}
\usepackage{changepage} 
\usepackage{lipsum} 

\usepackage{url}

\newcommand{\abs}[1]{\left|#1\right|}

\newcommand{\ur}{\textnormal{ur}}

\newcommand{\der}{\textnormal{der}}

\newcommand{\Tr}{\textnormal{Tr}}

\newcommand{\vol}{\textnormal{vol}}

\DeclareMathOperator{\modp}{\mathbf{mod}}

\DeclareMathOperator{\Rep}{Rep}

\DeclareMathOperator{\Ad}{Ad}

\DeclareMathOperator{\Frob}{Frob}

\DeclareMathOperator{\Gal}{Gal}
\DeclareMathOperator{\GL}{GL}
\DeclareMathOperator{\SL}{SL}

\DeclareMathOperator{\Irr}{Irr}
\DeclareMathOperator{\Ind}{Ind}

\DeclareMathOperator{\ind}{Ind}

\DeclareMathOperator{\fdeg}{fdeg}

\DeclareMathOperator{\Lie}{Lie}

\DeclareMathOperator{\triv}{triv}

\newtheorem{theorem}{Theorem}[section]
\newtheorem{conjecture}{Conjecture}[section]
\newtheorem{condition}{Condition}[section]
\newtheorem{lemma}[theorem]{Lemma}

\theoremstyle{definition}
\newtheorem{definition}[theorem]{Definition}
\theoremstyle{remark}
\newtheorem{remark}[theorem]{Remark}

\title{The Hiraga-Ichino-Ikeda conjecture for principal series of split p-adic groups}

\begin{document}

\author[G. Ricci]{Giulio Ricci}
\address{Radboud Universiteit Nijmegen, Heyendaalseweg 135, 6525AJ Nijmegen, the Netherlands}
\email{giulio.ricci@ru.nl}

\date{\today}

\begin{abstract}   
  Given a $p$-adic connected split reductive group $\mathcal{G},$ we use the local Langlands correspondence as defined in \cite{reeder2002isogenies} and \cite{PrinicipalDisconencted} to prove the HII conjecture for irreducible discrete series representations contained in a principal series of $\mathcal{G}$. We verify the predicted formula relating the formal degree of such representations to the adjoint $\gamma$-factor of their associated Langlands parameter. First, we prove it under the assumption that the center of $\mathcal{G}$ is connected, and then we generalize the result.  
\end{abstract}

\maketitle

\tableofcontents

\section{Introduction}

Let $\mathcal{G}$ be a connected reductive split group over a non-archimedean local field $K$ and write $G:=\mathcal{G}(K)$. Let $\mathcal{T}$ be a maximal split torus of $\mathcal{G}$ and write $T:=T(K).$ 

The aim of this paper is to prove the Hiraga-Ichino-Ikeda (HII) conjecture for discrete series representations occurring in the principal series of $G$, that is, those appearing as irreducible constituents of parabolically induced representations from characters of $T$.

In 1976, Borel showed that the category of representations of $G$ generated by their Iwahori-fixed vectors is equivalent to the category of modules over the Iwahori–Hecke algebra of $G$ \cite{borel1976admissible}. In 1987, Kazhdan and Lusztig classified the constituents of unramified principal series representations of $G$ and 
classified the simple modules over this Iwahori-Hecke algebra under the assumption that the center of $\mathcal{G}$ is connected \cite{ProofDeligneLanglandsConjecture}. Later on, in 2002, Reeder removed this connectedness assumption \cite{reeder2002isogenies}.

In 1988, Roche generalized these results, proving (under mild restrictions on the characteristic of $K$) that every Bernstein component associated to a principal series of $G$ can be described in terms of modules over a suitable Iwahori-Hecke algebra of a (possibly disconnected) group \cite{RochePrincipal}. Building on Roche's results, Reeder established a local Langlands correspondence for principal series representations, under the assumption that the center of $\mathcal{G}$ is connected \cite{reeder2002isogenies}. More recently, in 2017, Aubert, Baum, Plymen, and Solleveld completed (up to the same restrictions as Roche) the classification of all irreducible complex representations in the principal series of $G$, removing the assumption of connectedness of the center. This led to a local Langlands correspondence for principal series representations of $G$ \cite{PrinicipalDisconencted}.

The HII conjecture, formulated in 2007 by Hiraga, Ichino, and Ikeda \cite[Conjecture 1.4]{HII}, predicts a precise relationship between the formal degree of a discrete series representation and the adjoint $\gamma$-factor associated with its $L$-parameter, whenever a local Langlands correspondence is available. 

Let $G^{\vee}$ be the complex dual group of $G$, and write $Z(G)^s$ for the maximal $K$-split central torus in $G.$ If $\varphi$ is any $L$-parameter of $G^{\vee},$ we write $$S_{\varphi}^{\sharp}:=\pi_0(Z_{(G/Z(G)^s)^{\vee}}(\varphi))$$ for the component group of the centralizer of $\varphi$ in $(G/Z(G)^s)^{\vee}.$ Moreover, if $\psi$ is any order-0 additive character of $K,$ we can define the adjoint $\gamma$-factor $\gamma(0,\Ad_{G^{\vee}}\circ\varphi,\psi).$ For the exact definition of the $\gamma$-factor, we will refer to \cite{ArithmeticInv}.

 \begin{conjecture}\cite[1.4]{HII}\label{HII}
Assume we have a local Langlands correspondence for $G$. Let $\pi$ be an irreducible discrete series representation of $G$, and let $(\varphi_{\pi},\rho_{\pi})$ be the enhanced $L$-parameter associated to it. Let $\psi$ be an order-0 additive character of $K$. Then we have an equality $$\fdeg(\pi)=\frac{\dim(\rho_{\pi})}{\abs{S_{\varphi_\pi}^{\sharp}}}\abs{\gamma(0,\Ad_{G^{\vee}}\circ\varphi_{\pi},\psi)}.$$
 \end{conjecture}

This conjecture was recently proved for unipotent representations by Feng, Opdam and Solleveld \cite{OnFormDegrUnip}, and their work is the starting point for this paper. In fact, we combine their result with some of Roche's previous results, in order to prove the HII conjecture for principal series. In particular, we use Roche's equivalence of categories with Iwahori-Hecke algebras of (possibly disconnected) groups, to transfer the results from \cite{OnFormDegrUnip} to principal series. This leads to the main result of this paper:

\begin{theorem}
    Let $\mathcal{G}$ be a connected reductive $K$-split group and write $G:=\mathcal{G}(K).$ We consider the local Langlands correspondence for principal series representations as in \cite{PrinicipalDisconencted}. Then the HII conjecture \ref{HII} is true for every irreducible discrete series contained in a parabolically induced representation from a maximal torus of $G$.
\end{theorem}

\begin{center}
    \textbf{Acknowledgments}
\end{center}

Thanks to my supervisor Maarten Solleveld for his help and guidance throughout the development of this paper.

\section{Types and Hecke algebras for principal series}

In this section, we review the work of Roche \cite{RochePrincipal}. The most important theorem for us is Theorem \ref{Roche}, which will allow us to use the results in \cite{OnFormDegrUnip} to study principal series. 

We now set the notation for the remainder of this paper. Let $\mathcal{G}$ be a connected split reductive group over a non-archimedean local field $K$ and set $G:=\mathcal{G}(K)$. The ring of integers of $K$ will be denoted by $\mathcal{O}_K$, its maximal ideal by $\mathfrak{p}_K$ and its residue field by $k$. We denote by $W_K$ the Weil group of $K$, by $I_K$ its inertia subgroup and we write $q:=\abs{k}$. Moreover, we fix $\Frob\in W_K$ a geometric Frobenius element. Let $\mathcal{T} \subset \mathcal{B}$ be a maximal split torus and a Borel subgroup of $\mathcal{G}$, and write $T:=\mathcal{T}(K)$, $B:=\mathcal{B}(K)$ and $W=W(G,T)$ for the Weyl group of $G$. Let $\chi: T\to\mathbb{C}^{\times}$ be any character of $T$ and let $\varphi_{\chi}$ be the $L$-parameter associated to $\chi$. Finally, we will write $T^0$ for the maximal compact subgroup of $T$.

For the remainder of this paper, we impose the same condition on the characteristic of $k$ as in \cite[Page 179]{RochePrincipal}:

\begin{condition}\label{Condition}
    Let $\Phi:=\Phi(G,T)$ be the root system of $G$ associated to $T.$ Then, if $\Phi$ is irreducible we assume: \begin{itemize}
        \item If $\Phi$ is of type $A_n$, then $p>n+1.$
        \item If $\Phi$ is of type $B_n$, $C_n$ or $D_n$ then $p\neq 2.$
        \item If $\Phi$ is of type $F_4$, then $p\neq 2,3.$
        \item If $\Phi$ is of type $G_2$ or $E_6$, then $p\neq 2,3,5.$
        \item If $\Phi$ is of type $E_7$ or $E_8$, then $p\neq 2,3,5,7.$
    \end{itemize}

If $\Phi$ is not irreducible, we exclude primes attached to each of its irreducible factors.
\end{condition}

Let $\Rep(G)$ denote the category of smooth representations of $G$, and write $\Irr(G)$ for the set of equivalence classes of irreducible objects of $\Rep(G).$ We denote by $\mathcal{H}(G)$ the Hecke algebra of $G,$ so the convolution algebra of compactly supported, locally constant functions $f:G\to\mathbb{C}.$ We recall that there is a well-known equivalence between $\Rep(G)$ and the category of non-degenerate modules over $\mathcal{H}(G)$.

The category $\Rep(G)$ has a well-known decomposition: consider the pairs $(L,\sigma)$ where $L$ consists of the $F$-points of a Levi subgroup of $\mathcal{G}$ and $\sigma$ is an irreducible supercuspidal representation of $L$. Denote by $[L,\sigma]_G$ the set of pairs $(L',\sigma')$ such that $L'$ consists of the $F$-points of a Levi subgroup of $\mathcal{G}$, $\sigma'$ is an irreducible, supercuspidal representation of $L'$ and $(L',\sigma')=(gLg^{-1},g\sigma g^{-1}\otimes \nu)$ for some $g\in G$ and some unramified character $\nu$ of $L'.$ We call $[L,\sigma]_G$ an inertial equivalence class of $G$. The set of inertial equivalence classes of $G$ is called the \textbf{Bernstein spectrum} of $G$ and it is denoted by $\mathfrak{Be}(G).$

We say that a smooth irreducible representation $\pi\in \Rep(G)$ has inertial support $\mathfrak{s}=[L,\sigma]_G$ if $\pi$ appears as a subquotient of a representation parabolically induced from some element of the class $[L,\sigma]_G$. Define a full subcategory $\Rep(G)^{\mathfrak{s}}$ of $\Rep(G)$ as the category of representations for which each irreducible subquotient has inertial support $\mathfrak{s}.$ Then we have a decomposition $$\Rep(G)=\prod_{\mathfrak{s}\in\mathfrak{Be}(G)}\Rep(G)^{\mathfrak{s}}.$$

This is called the \textbf{Bernstein decomposition} of $G$. From now on, we are going to write $\mathfrak{s}_{\chi}=\mathfrak{s}:=[T,\chi]_G.$

\begin{definition}
    A \textbf{type} $t=(J,\sigma)$ for $G$ is a pair consisting of a compact open subgroup $J\subset G$ and an irreducible representation $\sigma$ of $J,$ satisfying the following:
    
    let $e_t\in\mathcal{H}(G)$ be the idempotent defined by  $$e_t(x)=\begin{cases}\frac{\Tr(\sigma(x^{-1}))}{\vol(J)}\text{ if } x\in J;\\
    0\hspace{1.5cm}\text{ else,} 
    \end{cases}$$
    let $R(G)_t$ be the full subcategory of $R(G)$ consisting of representations $(\pi, V)$ such that $\mathcal{H}(G)(e_tV)=V$, and let $\mathcal{H}_t=e_t\mathcal{H}(G)e_t$. Then the functor $R(G)_t\to \mathcal{H}_t\text{-}\modp$ given by $V\mapsto e_t(V)$ is an equivalence of categories. 

\end{definition}

A type $t$ is said to be a type for the inertial equivalence class $\mathfrak{s}$ (or just an $\mathfrak{s}$-type) if there is an equality between $R(G)_t$ and $R(G)^{\mathfrak{s}}$ as subcategories of $R(G).$  Roche constructed a $\mathfrak{s}$-type $(J_{\chi},\sigma_{\chi})$ as follows.

Given a smooth character $\lambda:\mathcal{O}_K^{\times}\to\mathbb{C}^{\times}$ the \textbf{conductor} of $\lambda$, denoted $\textnormal{cond}(\lambda)$, is defined as the least integer $n\geq 1$ such that $1+\mathfrak{p}_K^n\subset\ker(\lambda)$. For $\alpha\in\Phi,$ we define $c_{\alpha}=c_{\alpha}^{\chi}:=\textnormal{cond}(\chi\circ\alpha^{\vee}|_{\mathcal{O}_K^{\times}})$ and if $\chi\circ\alpha^{\vee}|_{\mathcal{O}_K^{\times}}=1$ we fix $c_{\alpha}=0.$  We define the function $f_{\chi}:\Phi\to\mathbb{Z}$ as

$$f_{\chi}(\alpha)=\begin{cases}\lfloor c_{\alpha}/2\rfloor\text{ if }\alpha\in\Phi^+\\ \min\{1,\lfloor(c_{\alpha}+1)/2\rfloor\}\text{ if }\alpha\in\Phi^-.
\end{cases}$$ 

The function $f_{\chi}$ is concave \cite[Lemma 3.4]{RochePrincipal}, and we can define the subgroups $$U_{f_{\chi}}=U_{\chi}:=\bigl\langle U_{\alpha,f_{\chi}(\alpha)}\bigl\rangle_{\alpha\in\Phi}$$ 
$$J_{\chi}:=\bigl\langle U_{\chi},T^0\bigl\rangle$$ 
$$T_{\chi}:=\bigl\langle\bigcup_{\alpha\in \Phi}\alpha^{\vee}(1+\mathfrak{p}_K^{c_{\alpha}})\bigl\rangle.$$

There is an isomorphism $J_{\chi}/U_{\chi}\cong T^0/T_{\chi}$ \cite[Lemma 3.2]{RochePrincipal} and we can inflate $\chi$ to a representation $\sigma_\chi$ of $J_{\chi}.$ Roche showed that $t_{\chi}=t:=(J_{\chi}, \sigma_{\chi})$ is a $\mathfrak{s}$-type. Therefore, we obtain an equivalence of categories between modules over the Hecke algebra $\mathcal{H}_t$ and representations of $G$ occurring in $\Ind_B^G(\chi).$ 

The key point for us is that $\mathcal{H}_t$ is isomorphic to an Iwahori-Hecke algebra of some other group $H$. If $\mathcal{I}_H$ is an Iwahori subgroup of a reductive split $K$-group $H,$ by the Iwahori-Hecke algebra of $H$ we mean the convolution algebra of compactly supported, $\mathcal{I}_H$-biinvariant functions $f:G\to\mathbb{C},$ denoted by $\mathcal{H}(H,\mathcal{I}_H)$.

We now outline the construction of $H$ given in \cite[\S 7]{RochePrincipal}. We will use the following notation:

    $$\Phi_{\chi}:=\{\alpha\in\Phi\mid\chi\circ\alpha^{\vee}(\mathcal{O}_K^{\times})=1\},$$
    $$\Psi_{\chi}:=(X^*(T),\Phi_{\chi}, X_*(T), \Phi_{\chi}^{\vee}),$$
    $$W_{\chi}:=\bigl\langle s_\alpha\in W\mid \alpha\in\Phi_{\chi}\bigl\rangle.$$

Let $H^\circ$ be the group of $K$-points of the reductive split group over $K$ corresponding to the root datum $\Psi_\chi.$ Notice that $H^{\circ}$ contains $T$ as a maximal torus. The set of positive roots in $\Phi^+\subset\Phi$ defined by $B,$ defines a set of positive roots $\Phi_{\chi}^+:=\Phi^+\cap\Phi_{\chi}$ in $\Phi_{\chi}$, a basis $\Delta_{\chi}$ of $\Phi_{\chi}$ and a Borel subgroup $B_H\subset H^{\circ}.$ Let $$C_{\chi}:=\{w\in W_{\chi}\mid w\Phi^+_{\chi}=\Phi^+_{\chi}\}.$$ This group naturally acts on $\Psi_{\chi}$ and preserves a basis, so we can embed it into the automorphisms of a pinning $(H^{\circ},B_H, \{x_{\alpha}\}_{\alpha\in\Delta_{\chi}}).$ This defines an action of $C_{\chi}$ on $H^{\circ},$ and we can form the semidirect product $$H:= H^{\circ}\rtimes C_{\chi}.$$ We denote the Iwahori subgroup of $H^{\circ}$ defined by the set of positive roots $\Phi_{\chi}^+$ by $\mathcal{I}_H$.

Let $T^{\vee}$ the torus dual to $T$. It is a well-known fact \cite[Theorem 2.2]{humphreysSemisim} that for any subset $A\subset T^{\vee},$ the identity component $Z_{G^{\vee}}(A)^{\circ}$ of the centralizer of $A$ in $G^{\vee},$ is generated by $T^{\vee}$ and those root subgroups $U_{\alpha^{\vee}}^{\vee}$ for which $\alpha^{\vee}(s)=1$ for every $s\in A.$ In our case, $Z_{G^{\vee}}(\varphi_{\chi}(I_K))^{\circ}$ is generated by $T^{\vee}$ and those root subgroups $U_{\alpha^{\vee}}^{\vee}\subset G^{\vee}$ for which $c_{\alpha}=0.$ This means that if $\Psi(Z_{G^{\vee}}(\varphi_{\chi}(I_K))^{\circ},T^{\vee})$ is the root datum of the reductive group $Z_{G^{\vee}}(\varphi_{\chi}(I_K))^{\circ}$ with respect to $T^{\vee},$ we have 
$$\Psi(Z_{G^{\vee}}(\varphi_{\chi}(I_K))^{\circ},T^{\vee})=\Psi(H^{\circ},T)^{\vee},$$ 
and we may identify $Z_{G^{\vee}}(\varphi_{\chi}(I_K))^{\circ}$ with the complex dual group of $H^{\circ}.$ Moreover, $C_{\chi}=\pi_0(H)$ is exactly $\pi_0(Z_{G^{\vee}}(\varphi_{\chi}(I_K)))$ and we can see $H^{\vee}:=Z_{G^{\vee}}(\varphi_{\chi}(I_K))={H^{\circ}}^{\vee}\rtimes C_{\chi}$ as a complex dual group for $H.$ 

We point out that in 2022, Kaletha generalized the local Langlands conjecture to a certain class of disconnected groups \cite{kaletha2022locallanglandsconjecturesdisconnected}. In his work, he generalizes the notion of Langlands dual and complex dual of a group to his setting. Using his definitions, the complex dual of $H$ is exactly $H^{\vee}.$

 We can now construct the Iwahori-Hecke algebras $\mathcal{H}(H,\mathcal{I}_H)$ and $\mathcal{H}(H^{\circ},\mathcal{I}_H)$ of $H$ and $H^{\circ}$ respectively. The action of $C_{\chi}$ on $H^{\circ}$ induces an action on $\mathcal{H}(H^{\circ},\mathcal{I}_H),$ and we can form the twisted tensor product $\mathcal{H}(H^{\circ},\mathcal{I}_H)\rtimes C_{\chi}.$ There is an isomorphism \cite[\S 8]{RochePrincipal} $$\mathcal{H}(H,\mathcal{I}_H)\cong \mathcal{H}(H^{\circ},\mathcal{I}_H)\rtimes C_{\chi}.$$

\begin{theorem}\cite[Theorem 8.2 and \S 10]{RochePrincipal}\label{Roche}
    Let $t=(J_{\chi},\sigma_\chi)$ as above. Then there are isomorphisms $$\mathcal{H}_t\cong\mathcal{H}( {H},\mathcal{I}_H)\cong\mathcal{H}(H^{\circ},\mathcal{I}_H)\rtimes C_{\chi}$$ such that the corresponding equivalences of module categories preserve square-integrabi\-lity and formal degrees.
\end{theorem}

Roche went on to study the group $C_{\chi}$ and proved that $C_{\chi}$ can be embedded in a quotient of a maximal compact subgroup of a certain torus, proving that $C_{\chi}$ is always abelian. Moreover, this quotient is trivial whenever the center of $\mathcal{G}$ is connected. Therefore, under this hypothesis, $C_{\chi}$ is trivial and $H= H^{\circ}.$ He also proved that the converse is true as well, meaning that if the center of $\mathcal{G}$ is disconnected, then $C_{\chi}$ is non-trivial for some $\chi.$

\section{Langlands correspondence for principal series}

In this section, we summarize some results from \cite{PrinicipalDisconencted}. Recall that a $L$-parameter for $G$ is a continuous homomorphism $$\phi:W_K\times\SL_2(\mathbb{C})\to G^{\vee}$$ with some extra properties:\begin{itemize}
    \item $\phi|_{\SL_2(\mathbb{C})}$ is algebraic;
    \item $\phi(W_K)$ consists of semisimple elements.
\end{itemize}
We say that $\phi$ is \textbf{discrete} if it does not factor through the $L$-group of any proper Levi subgroup of $G^{\vee}$. Equivalently, $\phi$ is discrete if its centralizer in $G^{\vee}$ is finite modulo the center of $G^{\vee}$. 
An irreducible representation $\rho_{\phi}\in \Irr(\pi_0(Z_{G^{\vee}}(\phi)))$ of the component group of the centralizer of $\phi$ in $G^{\vee}$ is called an \textbf{enhancement} of $\phi.$ We denote the set $G^{\vee}$-conjugacy classes of enhanced $L$-parameters of $G^{\vee}$ by $\Phi_e(G).$

Reeder parametrized the irreducible subquotients of a principal series via what we call \textbf{Kazhdan-Lusztig-Reeder parameters} (KLR-parameters for short): let $\phi$ be a Langlands parameter. The variety $\mathcal{B}_{G^{\vee}}^{\phi}$ of Borel subgroups of $G^{\vee}$ containing the image of $\phi$ is non-empty if and only if the image $\phi(W_K)$ is contained in some maximal torus of $G^\vee$. If $\mathcal{B}_{G^{\vee}}^{\phi}\neq\emptyset,$ the centralizer $Z_{G^{\vee}}(\phi)$ acts on it. Let $B_2\subset \SL_2(\mathbb{C})$ denote the subgroup of the upper triangular matrices. Then we have an action $\pi_0(Z_{G^{\vee}}(\phi))$ on the homology $H_*(\mathcal{B}^{\phi(W_K\times B_2)}_{G^{\vee}},\mathbb{C})$. We call an irreducible representation $\rho$ of $\pi_0(Z_{G^{\vee}}(\phi))$ \textbf{geometric} if it appears in this homology.
We define a Kazhdan-Lusztig-Reeder parameter for $G^{\vee}$ to be a pair $(\phi, \rho)$ consisting of an $L$-parameter $\phi$ and an irreducible geometric representation $\rho$ of $\pi_0(Z_{G^{\vee}}(\phi))$. Note that this forces $\mathcal{B}_{G^{\vee}}^\phi$ to be non-empty.

We have an action of $G^{\vee}$ on KLR-parameters by $$ g\cdot (\phi, \rho) = (g\phi g^{-1},\rho\circ\Ad^{-1}_g).$$ 
The inertial equivalence class $\mathfrak{s}=[T,\chi]_G,$ depends only on $\chi|_{T^0}.$ That is why we are only interested in the collection of KLR-parameters $(\phi,\rho)$ for which $\phi|_{I_K}=\varphi_{\chi}|_{I_K}.$ We write this set as $\{\text{KLR-par}\}^{\mathfrak{s}}$. This collection is not stable under $G^{\vee}$-conjugation, but it is under $H^{\vee}$-conjugation.  

Reeder defined a local Langlands correspondence using KLR-parameters, but only under the assumption that the center of $\mathcal{G}$ is connected:

\begin{theorem}\cite[Theorem 1]{reeder2002isogenies}
Assume the center of $\mathcal{G}$ is connected. Then, under Condition \ref{Condition}, the set $\Irr(G,B)$ of irreducible representations of $G$ which are subquotients of representations induced from characters of $B$ is in bijection with the set of $G^{\vee}$-conjugacy classes of KLR-parameters   
\end{theorem}

Reeder generalizes the classification of simple $\mathcal{H}(H^{\circ},\mathcal{I}_H)$-module done in \cite{ProofDeligneLanglandsConjecture} to the case in which the derived subgroup $H_{\der}^{\vee}$ of ${H^{\circ}}^{\vee}$ is not simply-connected \cite[Theorem 2]{reeder2002isogenies}. He proves in fact that, if $M$ is any connected reductive complex algebraic group and $\mathcal{M}$ is the affine Hecke algebra whose root datum is that of $M$ with constant parameter $q$ (or any non-root of unity in $\mathbb{C}^{\times}$), then the simple $\mathcal{M}$-modules are in bijection with $M$-conjugacy classes of triples $(t,x,\rho)$, where $t\in K$ is semisimple, $txt^{-1}= x^q$, and $\rho$  is an irreducible representation of $Z_M(t,x)$ appearing in the homology of $B_M^{t,x}.$ 

In 2017, Aubert-Baum-Plymen-Solleveld generalized these results via some geometric arguments \cite{PrinicipalDisconencted}. First, they defined Kazhdan–Lusztig triples for the disconnected groups. Let $M$ a connected reductive complex group and let $\Gamma$ a finite group of pinned automorphisms. 

\begin{definition}
        A Kazhdan–Lusztig triple $(t,x,\rho)$ for $\mathcal{H}(M)\rtimes \Gamma$ (or just a Kazhdan-Lusztig triple for $M\rtimes \Gamma$) consists of: \begin{itemize}
        \item A semisimple element $t\in M$ and a unipotent element $x\in M$ such that $txt^{-1}=x^q.$
        \item An irreducible representation $\rho$ of the component group $\pi_0(Z_{M\rtimes \Gamma}(t,x))$, such that every irreducible subrepresentation of the restriction of $\rho$ to $\pi_0(Z_{M}(t,x))$ appears in $H_*(\mathcal{B}_M^{t,x}; \mathbb{C}).$
    \end{itemize}
\end{definition}

Whenever $\Gamma$ is trivial, we have the usual definition of Kazhdan-Lusztig triple for connected groups. 

\begin{theorem}\label{KLtriple}\cite[Theorem 9.1]{PrinicipalDisconencted} There exists a natural bijection between $\Irr(\mathcal{H}(H^{\circ},\mathcal{I}_H)\rtimes C_{\chi})$ and $H^{\vee}$-conjugacy classes of Kazhdan–Lusztig triples.    
\end{theorem}

The proof of this theorem builds on \cite{reeder2002isogenies}. They in fact construct Kazhdan-Lusztig triples $(t_q,x,\rho)$ for $M\rtimes \Gamma$ by taking a Kazhdan-Lusztig triple $(t_q,x,\rho')$ for $M,$ and then extending $\rho'.$

Since every Kazhdan-Lusztig triple can be lifted to a KLR-parameter \cite[Lemma 7.1]{PrinicipalDisconencted} this bijection yields a LLC $$\Irr(\mathcal{H}_t)\to\Irr(\mathcal{H}(H,\mathcal{I}_{H}))\to\{\text{KLR-par}\}^{\mathfrak{s}}/H^{\vee}\to \Phi_e(G),$$

again assuming Condition \ref{Condition}.

\section{The HII conjecture for principal series}

In this section we prove conjecture \ref{HII} for principal series representations. We start by recalling the definition of the formal degree. 

Let $A$ be a closed subgroup of $Z(G)$ such that $Z(G)/A$ is compact, let $\mu_{G/A}$ be a Haar measure on $G/A$ and let $(\pi,V)$ be any smooth representation of $G$. The matrix coefficient $\pi_{v,w}$ of $\pi$ with respect to $v\in V$ and $w$ in the smooth dual $V^*$ of $V$, is the function $\pi_{v,w}: G\to\mathbb{C}$ defined by $$\pi_{v,w}(x)=\bigl\langle\pi(x)v,w\bigl\rangle.$$ 
Since $\pi$ is irreducible, there is a character $\lambda$ of $Z(G)$, called the central character of $\pi$, such that $\pi(z) =\lambda(z)$ for all $z\in Z(G)$. If the central character is unitary, the function $g\mapsto\abs{\pi_{v,w}(g)}^2$ is constant on cosets of $A$, and hence defines a function on $G/A$. We say that $(\pi, V)$ is a \textbf{discrete series} (with respect to $A$) if it has a unitary central character and $$\int_{G/A}\abs{\pi_{v,w}(g)}^2d\mu_{G/A}(g)<\infty$$ for all $v\in V,$ $w\in V^*$ (notice that this condition is independent of $A$, and also of the choice of a Haar measure on $G/A$).

It can be shown that every discrete series representation is unitary, in the sense that it admits a positive-definite $G$-invariant Hermitian product. The resulting isomorphism between $V$ and $V^*$ defines a matrix coefficient $\pi_{v,w}$ for $v,w\in V$. We write $\pi_v=\pi_{v,v}.$

\begin{definition}
    Let $(\pi, V)$ be a discrete series representation of $G$. As in \cite[\S 1]{HII}, we fix $A= Z(G)^s$ the maximal $K$-split central torus in $G$. Let $\mu_{G/A}$ be a Haar measure on $G/A$. The \textbf{formal degree} $\fdeg(\pi)=\fdeg(\pi,\mu_{G/A})$ of $\pi$ is the unique number (depending on $\mu_{G/A}$) such that for all $v\in V,$$$\int_{G/A}\abs{\pi_v(g)}^2\hspace{0.1cm}d\mu_{G/A}(g)=\frac{\abs{v}^2}{\fdeg(\pi)}.$$
\end{definition}

For the remainder of this paper, we fix the same Haar measure as in \cite{HII} and \cite[\S 1, Page 6]{OnFormDegrUnip} $\mu_G$ and $\mu_{H^{\circ}}$ on $G$ and $H^\circ$ respectively. Moreover, we fix a measure $\mu_H$ on $H$ by giving $C_{\chi}$ the counting measure. Let $\pi\subset \Ind_B^G(\chi)$ be an irreducible discrete series, let $\varphi=\varphi_{\chi}$ be the $L$-parameter associated to $\chi$ and $\varphi_\pi$ the $L$-parameter associated to $\pi.$ We denote by $\pi_G$ the corresponding $\mathcal{H}_t$-module, by $\pi_H$ the associated $\mathcal{H}(H, \mathcal{I}_H)$-module, and by $\pi'$ the corresponding representation of $H$. We refer to \cite[4.2.2]{OpdamSpectralCorresp} for the definition of the formal degree of a module over a Hecke algebra.

We want to compute the formal degree of $\pi$ by computing the formal degree of $\pi'.$ The formal degree of $\pi$ and the formal degree of $\pi_G$ differ by a factor $\vol(J_{\chi}).$ In fact, in order to make the obvious embedding $\mathcal{H}_t\to \mathcal{H}(G)$ trace-preserving, we need to normalize it so that the identity element of $\mathcal{H}_t$ is sent to the identity element of $\mathcal{H}(G).$ For more information about the Plancherel measure on $\mathcal{H}(G)$ and $\mathcal{H}_t$ we refer to \cite[\S3 and \S4]{TypesPlancherel}. Analogously, the formal degree of $\pi_H$ and of $\pi'$ differ just by a factor of $\vol(\mathcal{I}_H)^{-1}.$ Using \ref{Roche}, we get the formula $$\fdeg(\pi)=\frac{\fdeg(\pi_G)}{\vol(J_{\chi})}=\frac{\fdeg(\pi_H)}{\vol(J_{\chi})}=\frac{\vol(\mathcal{I}_H)}{\vol(J_{\chi})}\fdeg(\pi').$$

\begin{remark}
    Notice that we can compute the volume of $J_{\chi}$ in terms of the Iwahori subgroup $\mathcal{I}=\bigl\langle T^0,\{U_{\alpha,0}\}_{\alpha\in\Phi}\bigl\rangle:$ 

    $$\vol(J_{\chi})=\frac{\vol(\mathcal{I})}{[I:J_{\chi}]}=\frac{\vol(\mathcal{I})}{q^{\sum_{\alpha\in\Phi}f_{\chi}(\alpha)}}.$$ Moreover, $f_{\chi}(\alpha)+f_{\chi}(-\alpha)=c_{\alpha},$ therefore we have $$\vol(J_{\chi})=\frac{\vol(\mathcal{I})}{q^{\sum_{\alpha\in\Phi^+}c_{\alpha}}}.$$
\end{remark}

We now proceed to analyze the adjoint $\gamma$-factor. We write $\mathfrak{g}^{\vee}:=\Lie(G^{\vee}),$ we fix $\psi$ an order-0 additive character of $K$, and we recall from \cite{Tate} the following meromorphic functions in $s\in\mathbb{C}$: $$L(s,\Ad_{G^{\vee}}\circ \varphi)=\det\left(1-q^{-s}(\Ad_{G^{\vee}}\circ\varphi(\Frob)|_{(\mathfrak{g}^{\vee})^{\varphi(I_K)}})\right)^{-1},$$
$$\gamma(s,\Ad_{G^{\vee}}\circ\varphi,\psi)=\varepsilon(s,\Ad_{G^{\vee}}\circ\varphi,\psi)\frac{L(1-s,\Ad_{G^{\vee}}\circ \varphi)}{L(s,\Ad_{G^{\vee}}\circ\varphi)},$$ where $\varepsilon$ is the local factor defined in \cite[4.1.6]{Tate}. In order to study the $\gamma$-factor, we will split it into an unramified part $\gamma(0,\Ad_{G^{\vee}}\circ\varphi|_{(\mathfrak{g}^{\vee})^{\varphi(I_K)}},\psi)$ and a ramified part $\gamma(0,\Ad_{G^{\vee}}\circ\varphi|_{((\mathfrak{g}^{\vee})^{\varphi(I_K)})^{\perp}},\psi).$

Consider the idempotent $$E:=\int_{I_K}\frac{\varphi(w)}{\vol(I_K)}dw\in\mathfrak{gl}(\mathfrak{g}^{\vee}).$$ We have $E\cdot\mathfrak{g}^{\vee}=(\mathfrak{g}^{\vee})^{\varphi(I_K)},$ and we write $$((\mathfrak{g}^{\vee})^{\varphi(I_K)})^{\perp}:=(1-E)\cdot\mathfrak{g}^{\vee}.$$ We consider the decomposition $$\mathfrak{g}^{\vee}=(\mathfrak{g}^{\vee})^{\varphi(I_K)}\oplus ((\mathfrak{g}^{\vee})^{\varphi(I_K)})^{\perp},$$ and we compare it with 
$$\mathfrak{g}^{\vee}=\Lie(T^{\vee})\oplus\bigoplus_{\alpha^{\vee}\in\Phi(G^{\vee},T^{\vee})}\mathfrak{g}^{\vee}_{\alpha^{\vee}}.$$
Since $\varphi(W_K)\subset T^{\vee},$ $\Lie(T^{\vee})$ is necessarily contained in $(\mathfrak{g}^{\vee})^{\varphi(I_K)}.$ On the other hand, for $\mathfrak{g}^{\vee}_{\alpha^{\vee}}$ to be in $(\mathfrak{g}^{\vee})^{\varphi(I_K)}$ we need 
$$\alpha^{\vee}(g)X=\Ad_g(X)=X\text{ for every }g\in\varphi(I_K),\hspace{0.1cm} X\in\mathfrak{g}^{\vee}_{\alpha^{\vee}},$$
meaning that $\alpha^{\vee}\circ\varphi(x)$ acts trivially for every $x\in I_K.$ But this happens only if $\alpha\circ\chi(\mathcal{O}_K^{\times})=1,$ meaning that $c_{\alpha}=0.$ So, if $c_{\alpha}\neq 0, $ then $\mathfrak{g}^{\vee}_{\alpha^{\vee}}\subset ((\mathfrak{g}^{\vee})^{\varphi(I_K)})^{\perp}.$

Recall that the adjoint $\gamma$-factor is additive, meaning that $$\gamma(0,\Ad_{G^{\vee}}\circ\varphi,\psi)=\gamma(0,\Ad_{G^{\vee}}\circ\varphi|_{((\mathfrak{g}^{\vee})^{\varphi(I_K)})^{\perp}},\psi)\cdot  \gamma(0,\Ad_{G^{\vee}}\circ\varphi|_{(\mathfrak{g}^{\vee})^{\varphi(I_K)}},\psi).$$ That's why we focus now on studying the ramified part $\Ad_{G^{\vee}}\circ\varphi|_{((\mathfrak{g}^{\vee})^{\phi(I_K)})^{\perp}}.$ 
In this case, the $L$-functions involved in the $\gamma$-factor are trivial and $\varepsilon$-factor is not. We get  $$|\gamma(0,\Ad_{G^{\vee}}\circ\varphi_{((\mathfrak{g}^{\vee})^{\varphi(I_K)})^{\perp}},\psi)|=|\varepsilon(0,\Ad_{G^{\vee}}\circ\varphi|_{((\mathfrak{g}^{\vee})^{\varphi(I_K)})^{\perp}},\psi)|.$$ The absolute value of the $\varepsilon$-factor was studied in \cite{ArithmeticInv}. They proved that if $(\tau,W)$ is any Weil-Deligne representation, we have $$\abs{\varepsilon(0,\tau,\psi)}=q^{a(W)/2},$$
where $a(W)\in\mathbb{Z}{\geq 0}$ is the \textbf{Artin conductor} of $W$ and it is defined as follows. The inertial image $D_0=\tau(I_K)$ is the Galois group of a finite extension $K'/K$. If $A'$ is the ring of integers of $K'$ with maximal ideal $P'$, then we define the normal subgroup $D_j\leq D_0$ as the kernel of the action of $D_0$ on $A'/(P')^{j+1}$. This gives the lower indexing ramification filtration $$D_0\geq D_1\geq D_2\geq\cdots\geq D_n=1.$$
We write $d_j=|D_j|$. Now the Artin conductor is defined by $$a(W) = \sum_{j\geq 0}\dim(W/W^{D_j})\frac{d_j}{d_0}.$$ 

In our situation, we have 
$$|\varepsilon(0,\Ad_{G^{\vee}}\circ\varphi|_{((\mathfrak{g}^{\vee})^{\varphi(I_K)})^{\perp}},\psi)|=q^{a(((\mathfrak{g}^{\vee})^{\varphi(I_K)})^{\perp})/2}.$$

Let now $V:=((\mathfrak{g}^{\vee})^{\varphi(I_K)})^{\perp}$ and $V_{\alpha}:=\mathfrak{g}^{\vee}_{\alpha^{\vee}}$ for $\alpha\in \Phi$ with $c_{\alpha}\neq 0.$ We have a $I_K$-stable decomposition $$V=\bigoplus_{\alpha\in\Phi,c_{\alpha}\neq 0}V_{\alpha},$$ and using the additivity of the Artin conductor we write $$a(V)=\sum_{\alpha\in\Phi,c_{\alpha}\neq0} a(V_{\alpha})=\sum_{\alpha\in\Phi,c_{\alpha}\neq 0}\sum_{j\geq 0}\dim(V_{\alpha}/V_{\alpha}^{D^{\alpha}_j})\frac{d^{\alpha}_j}{d^{\alpha}_0}$$ where $D_j^{\alpha}$ is the $j$-th lower indexing ramification subgroup of the image of $I_K$ in $\GL(V_{\alpha})$ and $d_j^{\alpha}$ is the cardinality of $D_j^{\alpha}.$ Notice that $\dim(V_{\alpha}/V_{\alpha}^{D_j^{\alpha}})$ is either 0 or 1. 

Now we just look at the one-dimensional character $\alpha^{\vee}\circ \varphi$ of $W_K,$ which has depth $c_{\alpha}+1.$ We want to use some results from \cite{LocalFields}. Let $K^{\ur}$ be a maximal unramified extension of $K$ and $L/K^{\ur}$ a Galois extension such that the image of $I_K$ in $T^{\vee}$ is isomorphic to $\Gal(L/K^{\ur}).$ Let $\varphi_{L/K^{\ur}}$ denote the Herbrand function of $L/K^{\ur}$ and write $c_{V_{\alpha}}$ for the largest integer such that $\Gal(L/K^{\ur})_{c_{V_{\alpha}}}$ is not trivial. Then, using \cite[Proposition VI.2.5 + Corollary VI.2.1']{LocalFields} we get that $$a(V_{\alpha})=\varphi_{L/K^{\ur}}(c_{V_\alpha})+1.$$ 

Recall that $\Gal(L/K^{\ur})_{c_{V_{\alpha}}}=\Gal(L/K^{\ur})^{\varphi_{L/K^{\ur}}(c_{V_{\alpha}})},$  where $\Gal(L/K^{\ur})^{\varphi_{L/K^{\ur}}(c_{V_{\alpha}})}$ denotes the higher indexing ramification subgroup. We know that the largest integer $i$ such that $\Gal(L/K^{\ur})^i$ is non-trivial is $c_{\alpha}$ (here we are tacitly using that the LLC for tori preserves depth, see \cite[Theorem 7.19]{YuOttawa}). Therefore, we get $$\varphi_{L/K^{\ur}}(c_{V_{\alpha}})=c_{\alpha}.$$ 

We just proved the following lemma:
\begin{lemma}\label{epsilon}
    With notation as above, we have an equality $$|\varepsilon(0,\Ad_{G^{\vee}}\circ\varphi|_{((\mathfrak{g}^{\vee})^{\varphi(I_K)})^{\perp}},\psi)|=q^{(\sum_{\alpha\in\Phi, c_{\alpha}\neq 0} c_{\alpha}+1)/2}=q^{(\sum_{\alpha\in\Phi^+, c_{\alpha}\neq 0} c_{\alpha}+1)}.$$
\end{lemma}

We are now ready to prove our first theorem:

\begin{theorem}\label{Connected}
    Assume that the center of $\mathcal{G}$ is connected. Then $$\fdeg(\pi)=\frac{\dim(\rho_\pi)}{\abs{S^{\sharp}_{\varphi_\pi}}}\abs{\gamma(0,\Ad_{G^{\vee}}\circ\varphi,\psi)},$$ meaning that conjecture \ref{HII} is true. 
\end{theorem}

\begin{proof}
    Recall the formula $$\fdeg(\pi)=\frac{\vol(\mathcal{I}_H)}{\vol(J_{\chi})}\fdeg(\pi'),$$ where $\pi'$ is the $H$-representation associated to $\pi.$ Since the center of $\mathcal{G}$ is connected, we have $H= H^{\circ}.$ Since we chose the Haar measure on $H$ as is in \cite{OnFormDegrUnip} we can use the validity of the HII conjecture for unipotent representations of connected reductive groups \cite[Theorem 5.7]{OnFormDegrUnip}, to get    
    $$\fdeg(\pi')=\frac{\dim(\rho_{\pi'})}{\abs{S_{\varphi_{\pi'}}^{\sharp}}}\abs{\gamma(0,\Ad_{H^{\vee}}\circ\varphi_{\pi'},\psi)}.$$ 

    Since $T$ is split, it can be written $$T\cong K^{\times}\otimes X_*(T)\cong (\mathcal{O}_K^{\times}\times\mathbb{Z})\otimes X_*(T)=T^0\times X_*(T)$$ and the character $\chi$ can be written as $\chi=\chi_{\ur}\chi',$ where $\chi_{\ur}$ is unramified, and $\chi':T^0\to\mathbb{C}^{\times}.$ Let $\Rep_{\mathcal{I}_H}(H)$ and $\Rep_{T^0}(T)$ be the full subcategories of $\Rep(H)$ and $\Rep(T)$ of representations generated by their $\triv_{\mathcal{I}_H}$ and $\triv_{T^{0}}$ isotypical component respectively. Using \cite[Theorem 9.4]{RochePrincipal} we have a commutative diagram 
\begin{equation}\label{RocheDiagram}
    \begin{tikzcd}
	{\Rep(G)^{\mathfrak{s}}} && {\Rep_{\mathcal{I}_H}(H)} \\
	\\
	{\Rep(T)^{\mathfrak{s}}} && {\Rep_{T^0}(T)}
	\arrow["\cong"{description}, from=1-1, to=1-3]
	\arrow[from=3-1, to=3-3]
	\arrow["{\Ind_B^G}"{description}, from=3-1, to=1-1]
	\arrow["{\Ind_{B_{\chi}}^H}"{description}, from=3-3, to=1-3]
    \end{tikzcd}
\end{equation}

where the bottom map sends the character $\chi'\nu$ to $\nu$, where $\nu$ is any unramified character of $T.$ In particular, it sends the character $\chi$ to $\chi_{\ur}$ and the $L$-parameter $\varphi_{\pi'}$ is an unramified $L$-parameter, which is trivial on the inertia $I_K$ and equal to $\varphi_\pi$ outside of it.

Now we need to understand $S^{\sharp}_{\varphi_{\pi'}}=\pi_0(Z_{(H/Z(H)^s)^{\vee}}(\varphi_{\pi'})).$ Notice that since the group $G$ is split, $(G/Z(G)^s)^{\vee}=G^{\vee}_{\der}$ where $G^{\vee}_{\der}$ is the derived subgroup of $G^{\vee}$. Obviously, the same is true for $H.$ We are going to show that $$H^{\vee}_{\der}=H^{\vee}\cap G^{\vee}_{\der}.$$

The discreteness of $\varphi_\pi$ implies that we have an equality $\dim(Z(G^{\vee}))=\dim(Z(H^{\vee})):$ since $Z_{G^{\vee}}\left(\varphi_\pi)=Z_{H^{\vee}}(\varphi_\pi(\Frob),\varphi_\pi(\SL_2(\mathbb{C}))\right),$ the center $Z(H^{\vee})$ is in $Z_{G^{\vee}}(\varphi_\pi).$ We know that both $Z_{G^{\vee}}(\varphi_\pi)/Z(G^{\vee})$ and $Z_{G^{\vee}}(\varphi_\pi)/Z(H^{\vee})$ are finite, meaning that $$\dim(Z(G^{\vee}))=\dim(Z(H^{\vee})).$$ 

Therefore $H^{\vee}/Z(G^{\vee})$ is semisimple and we have equalities of Lie algebras $$\Lie(H^{\vee})=\Lie(H^{\vee}_{\der})\oplus\Lie(Z(G^{\vee})),$$ $$\Lie(G^{\vee})=\Lie(G^{\vee}_{\der})\oplus\Lie(Z(G^{\vee})).$$

But this implies that $$\Lie(H_{\der}^\vee)=\Lie(H^{\vee})\cap \Lie(G^{\vee}_{\der}).$$ Since $H^{\vee}_{\der}$ is obviously contained in $H^{\vee}\cap G^{\vee}_{\der},$ the equality of Lie algebras gives us the desired equality. We can check

\begin{align*}
S^{\sharp}_{\varphi_{\pi'}} 
&= \pi_0\left(Z_{H_{\der}^{\vee}}(\varphi_{\pi'})\right) 
= \pi_0\left(Z_{G_{\der}^{\vee}\cap H^{\vee}}(\varphi_{\pi'})\right)\\
&=\pi_0\left(Z_{Z_{G^{\vee}_{\der}}(\varphi_{\chi}(I_K))}(\varphi_{\pi'}(\Frob),\varphi_{\pi'}(\SL_2(\mathbb{C}))\right)
= \pi_0\left(Z_{G^{\vee}_{\der}}(\varphi_{\pi})\right) = S^{\sharp}_{\varphi_{\pi}}.
\end{align*}

By construction, the Lie algebra of $H^{\vee}$ is $(\mathfrak{g}^{\vee})^{\varphi(I_K)}$, so we have an equality $$\gamma(0,\Ad_{H^{\vee}}\circ\varphi_{\pi'},\psi)=\gamma(\Ad_{G^{\vee}}\circ\varphi_{\pi}|_{(\mathfrak{g}^{\vee})^{\varphi(I_K)}},\psi).$$

The enhancement $\rho_\pi$ is defined by Reeder to be precisely $\rho_{\pi'}$. Therefore $$\dim\rho_{\pi}=\dim\rho_{\pi'}.$$

The only remaining step, is to compare the $\varepsilon$-factor with the ratio $\vol(\mathcal{I}_H)/\vol(J_{\chi}).$ We already know that 

\begin{equation}\label{eq1}
    \vol(J_{\chi})=\frac{\vol(\mathcal{I})}{[\mathcal{I}:J_{\chi}]}=\frac{\vol(\mathcal{I})}{q^{\sum_{\alpha\in\Phi^+,c_{\alpha}\neq 0}c_{\alpha}}}.
\end{equation}
So, we only need to compute $\vol(\mathcal{I}_H)/\vol(\mathcal{I}).$ Let $\overline{\mathcal{I}}$ and $\overline{\mathcal{I}_H}$ be the $k$-points of the maximal reductive quotients of the group obtained by reducing modulo $\mathfrak{p}_k$  the Iwahori subgroups $\mathcal{I}$ and $\mathcal{I}_H$ respectively. We are normalizing the Haar measures of $G$ and $H$ so that we have $$\vol(\mathcal{I})=q^{-(\dim(G)+\dim(\overline{\mathcal{I}})/2}\cdot \abs{\overline{\mathcal{I}}},$$ $$\vol(\mathcal{I}_H)=q^{-(\dim(H)+\overline{\mathcal{I}_H})/2}\cdot \abs{\overline{\mathcal{I}_H}}.$$ Since both $H$ and $G$ contain the same maximal split torus $T,$ we have $$\abs{\overline{\mathcal{I}}}=\abs{\overline{\mathcal{I}_H}},\hspace{0.2cm}\dim(\overline{\mathcal{I}})=\dim(\overline{\mathcal{I}_H}).$$ Therefore we get 
\begin{equation}\label{eq2}
    \frac{\vol(\mathcal{I}_H)}{\vol(\mathcal{I})}=q^{(\dim(G)-\dim(H))/2}=q^{|\{\text{Positive roots with }c_{\alpha}\neq 0\}|}.
\end{equation}
Using \ref{eq1}, \ref{eq2} and Lemma \ref{epsilon} we have the equality

\begin{equation}\label{eq3}
|\varepsilon(0,\Ad\circ\varphi|_{(\mathfrak{g}^{\varphi(I_K)})^{\perp}},\psi)|=\frac{\vol(\mathcal{I}_H)}{\vol(J_{\chi})}.
\end{equation}

Summarizing, we proved  

\begin{align*}
\fdeg(\pi) 
&= \frac{\vol(\mathcal{I}_H)}{\vol(J_{\chi})}
   \cdot \frac{\dim(\rho_{\pi'})}{\abs{S_{\varphi_{\pi'}}^{\sharp}}}
   \cdot \abs{\gamma(0,\Ad\circ\varphi_{\pi'},\psi)} \\
&= \abs{\varepsilon\left(0,\Ad\circ\varphi|_{((\mathfrak{g}^{\vee})^{\varphi(I_K)})^{\perp}},\psi\right)}
   \cdot \frac{\dim(\rho_{\pi})}{\abs{S_{\varphi_\pi}^{\sharp}}}
   \cdot \abs{\gamma\left(0,\Ad\circ\varphi|_{(\mathfrak{g}^{\vee})^{\varphi(I_K)}},\psi\right)} \\
&= \frac{\dim(\rho_{\pi})}{\abs{S_{\varphi_\pi}^{\sharp}}}
   \cdot \abs{\gamma(0,\Ad\circ\varphi|_{(\mathfrak{g}^{\vee})},\psi)},
\end{align*}

where the first equality follows from the discussion at the beginning of page 8 and \cite{OnFormDegrUnip}, the second equality follows from the earlier argument in this proof, and the third equality follows from the additivity of the $\gamma$-factor.  
\end{proof}

We now drop the assumption on the connectedness of the center of $\mathcal{G}$, so that $H\neq H^{\circ}$. In \cite[Page 28-30]{PrinicipalDisconencted} modules of Hecke algebras like $\mathcal{H}(H,\mathcal{I}_H)$ were studied. In particular, they find a canonical way to extend a $\mathcal{H}(H^{\circ},\mathcal{I}_H)$-module $\nu$ to a module over $\mathcal{H}(H^{\circ},\mathcal{I}_H)\rtimes C_{\nu}$ where $C_{\nu}$ is the stabilizer of $\nu$ in $C_{\chi}.$ Using this, we find that the representation $\pi'$ can be obtained by first forming the tensor product of an irreducible representation $\nu_\pi$ of $H^{\circ}$ (extended to $H^\circ\rtimes C_{\nu_\pi}$) with an irreducible representation $\sigma$ of $C_{\nu_\pi}$ (inflated to $H^{\circ}\rtimes C_{\nu_\pi}$), and then inducing this tensor product to $H.$ Notice that since $C_{\chi}$ is abelian, $\sigma$ is 1-dimensional. 

We now relate the formal degree of $\pi'$ with the one of $\nu_\pi.$ 

\begin{lemma}\label{Lemma 1}
    The finite-index induction $\pi':=\ind_{H_{\nu_\pi}}^{H}(\nu_\pi\otimes\sigma)$ has the same formal degree of $\nu_\pi\otimes \sigma.$
\end{lemma}

\begin{proof}
    This is a well-known fact. Let $W$ be the space of $\nu_\pi\otimes \sigma$ and let $V$ be the space of the induction. Let $A\subset Z(H)$ a subgroup such that $Z(H)/A$ is compact, let $\mu_{H/A}$ a Haar measure on $H/A$ and let $\mu_{H_{\nu_\pi}/A}$ the restriction of $\mu_{H/A}$ to $H_{\nu_\pi}/A$. 
    
    We define an isometric embedding $W\to V$ as follows. Given a vector $w\in W$, define $\tilde{w}\in V$ by $\tilde{w}(x)=1_{H_{\nu_\pi}}(x)(\nu_\pi\otimes \sigma)(x)w$. The space $H_{\nu_\pi}\backslash H$ is finite and we take the counting measure $\mu_{H_{\nu_\pi}\backslash H}$ on it. With this choice, the map $W\to V$ defined by $w\mapsto \tilde{w}$ is an isometric embedding. This means that the matrix coefficient of $\tilde{w}$ is the extension by zero of the matrix coefficient of $w$ and the $L^2$-norms of $w$ and $\tilde{w}$ coincide. Then, for any nonzero $w\in W$ we have 
    \begin{align*}
&\fdeg\left(\pi', \mu_{H/A}\right)=\frac{|\tilde{w}|^2}{\int_{H/A} |\pi'_{\tilde{w}}|^2 \, d\mu_{H/A}} \\
&\quad= \frac{|w|^2}{\int_{H_{\nu_\pi}/A} |(\nu_\pi \otimes \sigma)_w|^2 \, d\mu_{H_{\nu_\pi}/A}} =\fdeg\left(\nu_\pi \otimes \sigma, \mu_{H_{\nu_\pi}/A}\right).
\end{align*}

\end{proof}

\begin{lemma}\label{Lemma 2}
    We have an equality $$\fdeg(\nu_\pi\otimes\sigma)\abs{C_{\nu_\pi}}=\fdeg(\nu_\pi).$$
\end{lemma}

\begin{proof}
    Let $V$ be the vector space of the representation $\nu_{\pi}$, $W=\mathbb{C}$ the space of $\sigma$ and choose any element $w\in \mathbb{C}$ such that $|w|^2=1.$

    The norm of $v\otimes w\in V\otimes W$ is equal to the norm of $v$ for every $v\in V.$  In order to compare the formal degrees, it suffices to compare the integrals of the matrix coefficients. Note that we may assume \( Z(G) \subset T \subset H \), and that \( Z(H) / Z(G) \) is compact \cite[Theorem 10.7]{RochePrincipal}. Again, we can choose an Haar measure $\mu_{H_{\nu_\pi}/Z(G)}$ on $H_{\nu_\pi}/Z(G)$ and consider its restriction $\mu_{H^{\circ}/Z(G)}$ to $H^{\circ}/Z(G).$ We have

\begin{align*}
&\int_{H_{\nu_\pi} / Z(G)}
\left| \left\langle (\nu_\pi \otimes \sigma)(hx)(v \otimes w), v \otimes w \right\rangle \right|^2
\, d\mu_{H_{\nu_\pi} / Z(G)}(hx) \\
&\quad= \int_{{H^{\circ}}\rtimes C_{\nu_\pi} / Z(G)}
\left| \left\langle \nu_\pi(h)v, v \right\rangle \right|^2
\left| \left\langle \sigma(x)w, w \right\rangle \right|^2
\, d\mu_{{H^{\circ}} \rtimes C_{\nu_\pi} / Z(G)}(hx)
\end{align*}

Since $\abs{\sigma(x)}^2=1$ and $\bigl\langle w,w\bigl\rangle^2=1,$ we have 

\begin{align*}
\fdeg(\nu_\pi \otimes \sigma) 
&= \frac{\left|v \otimes w\right|^2}
         {\int_{H_{\nu_\pi}/ Z(G)}
         \left| \left\langle \nu_\pi(h)v, v \right\rangle \right|^2 
         \, d\mu_{H_{\nu_\pi} / Z(G)}(hc)} \\
&= \frac{|v|^2}
         {|C_{\nu_\pi}| \cdot 
         \int_{{H^{\circ}} / Z(G)} 
         \left| \left\langle \nu_\pi(h)v, v \right\rangle \right|^2 
         \, d\mu_{{H^{\circ}} / Z(G)}(h)} \\
&= \frac{\fdeg(\nu_\pi)}{|C_{\nu_\pi}|}\qedhere
\end{align*}

\end{proof}

Now we are ready to prove our last theorem:

\begin{theorem}
    We have an equality $$\fdeg(\pi)=\frac{\dim(\rho_\pi)}{\abs{S^{\sharp}_{\varphi_{\pi}}}}\gamma(0,\Ad_{G^{\vee}}\circ\varphi,\psi).$$ This means that the HII conjecture is true for irreducible discrete series appearing in a principal series.
\end{theorem}

\begin{proof}
We define $\pi',\nu_\pi$ and $\sigma$ as above. The $H^{\circ}$ representation $\nu_\pi$ is a unipotent, and using Lemma \ref{Lemma 1}, Lemma \ref{Lemma 2} and \cite{OnFormDegrUnip} we obtain 
$$\fdeg(\pi')=\frac{\dim(\rho_{\nu_\pi})}{\abs{C_{\nu_\pi}}\abs{S_{\varphi_{\nu_\pi}}^{\sharp}}}\abs{\gamma(0,\Ad_{H^{\vee}}\circ\varphi_{\nu_\pi},\psi)},$$
and again, thanks to diagram \ref{RocheDiagram}, we have $\varphi_{\nu_\pi}=\varphi_{\pi'}.$ In the proof of Theorem \ref{Connected}, we have already shown the equalities $$\frac{\vol(\mathcal{I}_H)}{\vol(J_{\chi})}=\varepsilon(0,\Ad_{G^{\vee}}\circ\varphi,\psi),$$
$$\gamma(0,\Ad_{G^{\vee}}\circ\varphi|_{(\mathfrak{g}^{\vee})^{\varphi(I_K)}},\psi)=\gamma(0,\Ad_{G^{\vee}}\circ\varphi_{\pi'},\psi).$$
Moreover, we have seen that  
$$Z_{{(H^{\circ}}/Z(H^{\circ})^s)^{\vee}}(\varphi_{\pi'}) = Z_{(G/Z(G)^s)^{\vee}}(\varphi_\pi).$$
The only thing that is left to prove is that 
$$\frac{\dim(\rho_{\pi})}{\abs{S^{\sharp}_{\varphi_{\pi}}}}=\frac{\dim(\rho_{\nu_\pi})}{\abs{S^{\sharp}_{\varphi_{\pi'}}}\cdot\abs{C_{\nu_{\pi}}}}.$$

The full centralizer $Z_{{H}^{\vee}}(\varphi_{\pi'})=Z_{G^{\vee}}(\varphi_\pi)$ is given by the semidirect product of its identity component $Z_{{H^{\circ}}^{\vee}}(\varphi_{\pi'})\subset Z_{G^{\vee}}(\varphi_{\pi})$ and the subgroup of $C_{\chi}$ which consists of the elements that preserve the $G^{\vee}$-conjugacy class of $\chi,$ but this is exactly $C_{\chi}$.

Therefore we get $$\abs{S^{\sharp}_{\varphi_{\pi'}}}\cdot\abs{C_{\chi}}=\abs{S^{\sharp}_{\varphi_{\pi}}}.$$
 Now we look at the enhancements. In particular, we look more carefully at the results in \S 3 from  \cite{PrinicipalDisconencted}. Their construction is compatible with normalized parabolic induction. This means that if $L$ is an $\mathcal{H}({H^{\circ}},\mathcal{I}_H)$-module, and $(\phi_L,\rho_L)$ is the associated KLR-parameter, then the induction $L':=\ind_{\mathcal{H}(H^{\circ},\mathcal{I}_H)\rtimes C_{\rho_L}}^{\mathcal{H}(H,\mathcal{I}_H)}L\otimes\sigma$ corresponds to the KLR-parameter $(\phi_L,\rho_{L'})$ with $\rho_{L'}=\ind_{\pi_0(Z_{{H^{\circ}}^{\vee}}(\varphi_{\pi'}))\rtimes C_{\rho_L}}^{\pi_0(Z_{G^{\vee}}(\varphi_\pi))}(\rho_L\otimes\sigma).$

In our case, this means that 
$$\rho_{\pi'}|_{\pi_0(Z_{{H^{\circ}}^{\vee}}(\varphi_{\pi'}))\rtimes C_{\nu_\pi}}=\bigoplus_{c\in C_{\nu_\pi}\backslash C_{\chi}}{}^c\rho_{\nu_\pi} $$ 
where $C_{\nu_\pi}$ is seen as the stabilizer in $C_{\chi}$ of $\varphi_{\pi'}.$ This means that $$\dim(\rho_\pi)=[C_{\chi}:C_{\nu_\pi}]\cdot \dim(\rho_{\nu_\pi}).$$

Therefore we have 
\begin{align*}
\fdeg(\pi) 
&= \frac{\vol(\mathcal{I}_H)}{\vol(J_{\chi})} \fdeg(\pi')
= \frac{\vol(\mathcal{I}_H)}{\vol(J_{\chi})} \cdot \frac{\fdeg(\nu_\pi)}{|C_{\nu_\pi}|} 
= \frac{\dim(\rho_{\nu_\pi})}{|C_{\nu_\pi}| \cdot |S_{\varphi_{\pi'}}^{\sharp}|} \cdot \left| \gamma(0, \Ad_{G^{\vee}} \circ \varphi, \psi) \right| \\
&= \frac{\dim(\rho_{\pi})}{|C_{\nu_\pi}| \cdot |S_{\varphi_{\pi'}}^{\sharp}| \cdot [C_{\chi} : C_{\nu_\pi}]} \cdot \left| \gamma(0, \Ad_{G^{\vee}} \circ \varphi, \psi) \right|\\ 
&= \frac{\dim(\rho_{\pi})}{|S_{\varphi_{\pi}}^{\sharp}|} \cdot \left| \gamma(0, \Ad_{G^{\vee}} \circ \varphi, \psi) \right|,
\end{align*}

as desired.
\end{proof}

\bibliographystyle{alpha}
\nocite{*}
\bibliography{Bibliografia}

\end{document}